\documentclass[10pt]{article}
\usepackage{pgfplots}
\pgfplotsset{width=7cm,compat=1.8}
\usetikzlibrary{patterns}
\usepackage[latin1]{inputenc}
\usepackage{latexsym}
\usepackage{amsmath,amssymb,amsfonts,amsthm}
\usepackage{xcolor}
\usepackage{mathrsfs}
\usepackage{comment}
\usepackage{bbm}
\usepackage{paralist}

\usepackage{tikz}
\usetikzlibrary{calc}
\usetikzlibrary{decorations.pathreplacing}
\usetikzlibrary{arrows}
\usepgfplotslibrary{fillbetween}
\usepackage{mathtools}

\usepackage[toc,page]{appendix}
\usepackage{enumitem}

\usepackage[numbers,comma,sort]{natbib}

\definecolor{db}{RGB}{0, 0, 130}
\usepackage[colorlinks=true,citecolor=db,linkcolor=black,urlcolor=blue,pdfstartview=FitH]{hyperref}

\usepackage{pdfsync}

\definecolor{rp}{rgb}{0.25, 0, 0.75}
\definecolor{dg}{rgb}{0, 0.5, 0}

\newcommand{\apriori}{1+{\frac{-H\beta}{1+H(\beta-1)}}}
\newcommand{\aprioritwo}{1+{\frac{-H\beta^\prime}{1+H(\beta^\prime-1)}}}
\newcommand{\R}{\mathbb{R}}

\newcommand{\N}{\mathbb{N}}
\newcommand{\EE}{\mathbb{E}}

\newcommand{\bqn}{\begin{equation}}
\newcommand{\eqn}{\end{equation}}
\newcommand{\bqne}{\begin{equation*}}
\newcommand{\eqne}{\end{equation*}}
\newcommand{\blue}{\textcolor{black}}

\newcommand{\red}{\textcolor{black}}
\def \limbashaut#1#2#3{\mathrel{\mathop{\kern 0pt#1}\limits_{#2}^{#3}}}

\newcommand{\COO}[4]{\mathcal{C}^{#1}_{#4}(L^{#2,#3})}

\DeclareMathOperator{\supp}{supp}

\numberwithin{equation}{section}

\usepackage{todonotes}

\makeatletter
\newcommand{\customlabel}[2]{%
   \protected@write \@auxout {}{\string\newlabel {#1}{{#2}{\thepage}{#2}{#1}{}}}%
   \hypertarget{#1}{#2\hspace{-0.14cm}}
}
\makeatother

\newtheorem*{acknowledgement}{Acknowledgement}

\theoremstyle{definition}
\newtheorem{definition}{Definition}[section]

\theoremstyle{remark}
\newtheorem{remark}[definition]{Remark}

\theoremstyle{plain}

\newtheorem{lemma}[definition]{Lemma}
\newtheorem{proposition}[definition]{Proposition}
\newtheorem{theorem}[definition]{Theorem}
\newtheorem{assumption}[definition]{Assumption}

\author{Lukas Anzeletti\footnote{Universit\'e Paris-Saclay, CentraleSup\'elec, MICS and CNRS FR-3487;   \texttt{lukas.anzeletti@centralesupelec.fr}. Acknowledging the support of the Labex de Math\'ematiques Hadamard. The author also thanks UNICAMP and FAPESP for the financial support to participate at SPSAS2022.}
}

\title{ \Large{\textbf{A note on weak existence for SDEs driven by fractional Brownian motion}}}
\begin{document}

\maketitle

\begin{abstract}
We are interested in existence of solutions to the $d$-dimensional equation
\begin{equation*}
X_t=x_0+\int_0^t b(X_s)ds + B_t,
\end{equation*}
where $B$ is a (fractional) Brownian motion with Hurst parameter $H\leqslant 1/2$ and $b$ is 
an $\mathbb{R}^d$-valued \red{measure} in some Besov space. We exhibit a class of drifts $b$ such that weak existence holds. In particular existence of a 
weak solution is shown for $b$ being a finite $\mathbb{R}^d$-valued measure for any 
$H<1/(2d)$. 
\end{abstract}

\noindent\textit{\textbf{Keywords and phrases:}  Regularization by noise, fractional Brownian motion.} 

\medskip

\noindent\textbf{MSC2020 subject classification: } 60H50, 60H10, 60G22, 34A06.

\setcounter{tocdepth}{2}
\renewcommand\contentsname{}

\section{Introduction}

Throughout the paper we consider the $d$-dimensional stochastic differential equation (SDE)
\begin{align}
X_t&=x_0 + \int_0^t b(X_s)ds + B_t,\label{eq:skew}
\end{align}
where $x_0 \in \mathbb{R}^d$, $B$ is a fractional Brownian motion with Hurst parameter $H\leqslant 1/2$ and $b$ is \red{an $\mathbb{R}^d$-valued measure (i.e. each component is a measure)}. For the moment \eqref{eq:skew} is to be interpreted formally. 

In the Brownian case (i.e. $H=1/2$), there is an extensive literature for 
SDEs with irregular drift which we will 
not describe thoroughly. However, we point to the early work of Veretennikov 
\cite{Veretennikov} for bounded measurable drifts and the more general $L^p-L^q$ criterion 
of Krylov and R\"ockner \cite{KrylovRockner} for which the authors proved strong existence and pathwise uniqueness (both works allowing for time-dependent drift). For the case of possibly distributional drifts see \cite{FlandoliRussoWolf, BassChen, DelarueDiel,FIR}. 
\red{We also point out the work
\cite{HarShepp} and extensions thereof in \cite{LeGall} on a $1$-dimensional SDE }involving the local time at $0$ of the solution. This formally corresponds to a drift 
$b=a\delta_{0}$, where $a\in [-1,1]$ and $\delta_{0}$ is the Dirac distribution. The solution to such an equation is the so-called skew Brownian motion, see \cite{Lejay} for more 
details and various constructions. The case 
$a=1$ corresponds to reflection above $0$. \red{Additionally, it was shown in \cite{HarShepp,LeGall} that for $a$ with $|a|>1$, existence of solutions is lost.}

The above can also be considered as a special case of a second class of interesting problems, that is solving Equation \eqref{eq:skew} 
for a distribution $b$ and a fractional Brownian motion $B$ with sufficiently small Hurst 
parameter $H$. A first attempt in this direction is due to Nualart and Ouknine 
\cite{NualartOuknine}, who proved existence and uniqueness for a class of non-Lipschitz drifts.
\red{For $b=a\delta_{0}$ with $a\in \R$, Catellier and Gubinelli \cite{CatellierGubinelli} established the well-posedness of this equation for $H<1/(2d+2)$ (as a special case of a result in general Besov spaces) using nonlinear Young integrals.} \red{In \cite{AnRiTa} for $d=1$ and generalized to any dimension $d$ in \cite{Haressetal}, the condition for well-posedness was extended to $H\leqslant 1/(2d+2)$ and existence of a weak solution was shown to hold for $H<1/(2d+1)$ using the stochastic sewing lemma. Additionally, by \cite{AnRiTa}, there exists a weak solution for $H<\sqrt{2}-1$ in $1$-dimension, using nonlinear Young integrals in $p$-variation. However, even for $d=1$, both for existence and uniqueness we observe a gap between the Brownian case ($H=1/2$), with well-posedness for $|a|\leqslant 1$ proven in 
\cite{LeGall}, and the result for fractional Brownian motion with $H<\sqrt{2}-1$ (existence) and $H\leqslant 1/4$ (uniqueness).} The aim of this paper is to \red{close this gap in dimension $1$ in terms of weak existence}. For a more detailed summary of the literature for \eqref{eq:skew} in the fractional Brownian motion case as well as references to related works we refer the reader to Appendix~\ref{summary}.

During the preparation of this manuscript, the authors in \cite{Leetal} closed the gap mentioned above as a special case of a result for any dimension $d\geqslant 1$ (see \cite[Theorem 2.11]{Leetal}). Therein the authors prove existence of a weak solution to \eqref{eq:skew} for $b$ being a finite measure and $H<\frac{1}{d+1}$. \red{This seems to be the complete subcritical regime in any dimension \textendash{} for a scaling argument see \cite[Example 1.1]{GaleatiGerencser} and for a counterexample with a drift given by a singular function see \cite[Theorem 2.7]{Leetal}. In particular their result implies the following statement that can be deduced from Theorem~\ref{thm:weakexistence}.}

\begin{theorem}\label{cor:measure}
Let $b$ be a finite $\mathbb{R}^d$-valued measure and $H<1/(2d)$. Then there exists a weak solution to \eqref{eq:skew}.
\end{theorem}

Nevertheless, due to the simplicity of the proof, we state and prove 
Theorem~\ref{thm:weakexistence}. The general idea therein is to make use of the 
nonnegativity of the drift, giving certain \emph{a priori} estimates of the solution under loose 
assumptions on the regularity of the drift, respectively the Hurst parameter 
$H$, which is eventually leading to existence of a weak solution. The idea of the proof of 
Theorem~\ref{thm:weakexistence}, that resembles the ones in \cite{Atetal} and \cite{AnRiTa}, heavily relies on the stochastic sewing Lemma with random controls.

\section{Notations and definitions}\label{defnot}

Throughout the paper, we use the following notations and conventions:
\begin{compactitem}
	\item Constants $C$ might vary from line to line. 
\item
For topological spaces $X,Y$ we denote the set of continuous function  
from $X$ to $Y$ by $\mathcal{C}(X,Y)$.
\item For $x=(x_1,\cdots,x_d) \in \mathbb{R}^d$, let $|x|=\sum_{i=1}^d |x_i|$.
\item 
Let \(s < t\) be two real numbers and \(\Pi = (s = t_0 < t_1 < \cdots < t_n =t)\) be a 
partition of \([s,t]\), we denote $|\Pi| = \sup_{i=1,\cdots,n}(t_{i} - t_{i-1})$ the mesh 
of $\Pi$. 
\item For $s,t \in \mathbb{R}$ with $s\leqslant t$, we denote $\Delta_{[s,t]}:=\{(u,v):s\leqslant u \leqslant v \leqslant t\}$.
\item For any function $f$ defined on $[s,t]$, we denote $f_{u,v}:=f_v-f_u$ for $(u,v) \in \Delta_{[s,t]}$. 
\item For any function \(g\) defined on $\Delta_{[s,t]}$ and 
$s\leqslant r\leqslant u \leqslant v\leqslant t$, we denote $\delta 
g_{r,u,v}:=g_{r,v}-g_{r,u}-g_{u,v}$.
\item 
For a probability space $\Omega$ and $p \in [1,\infty]$, the norm on $L^p(\Omega)$ is denoted by $\|\cdot\|_{L^p}$.
\item 
We denote by $(B_t)_{t \geqslant 0}$ a fractional Brownian motion 
with  Hurst parameter $H\leqslant 1/2$. 
\item 
\red{Filtrations \((\mathcal{F}_t)_{t \geqslant 
	0}\) will be denoted by \(\mathbb{F}\).}
\item All filtrations are assumed to fulfill  the usual conditions. 
\item 
\red{For a filtration $\mathbb{F}$, we call $(W_t)_{t \geqslant 
	0}$ a $\mathbb{F}$-Brownian 
motion if $(W_t)_{t \geqslant 
	0}$ is a Brownian motion adapted to $\mathbb{F}$ and $W_t-W_s$ is independent of $\mathcal{F}_s$ for $0 \leqslant s \leqslant t$. For such filtration, the conditional expectation $\EE[\cdot \mid \mathcal{F}_s]$ is denoted by 
$\EE^s[\cdot]$.}
\end{compactitem}

For any $t>0$ and $x \in \mathbb{R}^d$, let $g_t(x):=\tfrac{1}{(2\pi 
t)^{d/2}}e^{-\tfrac{|x|^2}{2t}}$. For $\phi:\mathbb{R}^d\rightarrow \mathbb{R}^d$,
let
\begin{align} \label{Gaussiansemigroup}
    G_t\phi(x):=(g_t * \phi) (x).
\end{align}

A continuous function $w:\Delta_{[s,t]}\rightarrow [0,\infty)$ is a control function if, for 
$s\leqslant r \leqslant u \leqslant v \leqslant t$,
\begin{align*}
    w(r,u)+w(u,v) \leqslant w(r,v),
\end{align*}
and $w(r,r)=0$ for all $r \in [s,t]$. 

We call a measurable function $\lambda \colon \Delta_{[s,t]}\times \Omega\rightarrow \mathbb{R}_{+}$ a random control if there exists a set $\Omega^\prime$ of full measure such that for $\omega \in \Omega^\prime$, $\lambda(\cdot,\omega)$ is a control.

We denote the nonhomogeneous Besov spaces by $\mathcal{B}^s_p$. \red{For a precise definition of these spaces see Section~\ref{sec:Besov}}. For $s \in \mathbb{R}_{+}\setminus \mathbb{N}$ and $p=\infty$, Besov spaces 
coincide with H\"older spaces (see \cite[p.99]{BaDaCh}).

\begin{definition} \label{def:beta-}
Let $\beta \in \mathbb{R}$. We say that $(f_n)_{n \in \mathbb{N}}$ converges to $f$ in $\mathcal{B}_\infty^{\beta-}$ if $\sup_{n \in \mathbb{N}} \|f_n\|_{\mathcal{B}_\infty^\beta}<\infty$ and 
\begin{align*}
\forall \beta^\prime<\beta, \quad    \lim_{n \rightarrow 
\infty}\|f_n-f\|_{\mathcal{B}_\infty^{\beta^\prime}}=0.
\end{align*}
\end{definition}

\begin{remark} \label{rem:measure}
Throughout the paper we make use of the nonnegativity of the drift and therefore work with nonnegative distributions. However any such distribution is actually given by a \blue{Radon measure (see \cite[Exercise 22.5]{Treves}) and therefore we work directly with such measures.}
\end{remark}

\paragraph{\emph{Stochastic sewing.}}
Stochastic sewing was originally introduced in \cite{Le}. In Lemma~\ref{lem:stsrandomcontrol} we recall a recent extension (see \cite[Theorem 4.7]{Atetal}) involving random controls. This version of stochastic sewing allows us to make use of the nonnegativity of the drift as the corresponding integral will have nonnegative increments in each component. In particular it will give rise to a random control.

\begin{lemma} \label{lem:stsrandomcontrol}
Let $m \in [2,\infty)$ and $0\leqslant S<T$. Let $A: \Delta_{[S,T]} \rightarrow L^m$ be $\mathbb{R}^d$-valued such that $A_{s,t}$ is $\mathcal{F}_t$-measurable for any $(s,t) \in \Delta_{[S,T]}$. Let $\lambda$ be a random control. Assume that there exist constants $\Gamma_1,\Gamma_2,\alpha_1,\beta_1\geqslant 0$ and $\varepsilon>0$ such that $\alpha_1+\beta_1>1$ and
\begin{align}
|\EE^u\delta A_{s,u,t}|&\leqslant \Gamma_1|t-s|^{\alpha_1}\lambda(s,t)^{\beta_1} \text{ a.s.,}\label{sts1}\\
    \|\delta A_{s,u,t}\|_{L^m} &\leqslant \Gamma_2 (t-s)^{1/2+\varepsilon}, \label{sts2}
\end{align}
for every $(s,t) \in \Delta_{[S,T]}$ and $u \coloneqq (s+t)/2$. 
Assume there exists a process $(\mathcal{A}_t)_{t\in [S,T]}$ such that, for any $t \in [S,T]$ 
and any sequence of partitions $\Pi_k=\{t_i^k\}_{i=0}^{N_k}$ of $[S,t]$ with mesh size going 
to zero, we have
\begin{align} \label{sts3}
    \mathcal{A}_t=\lim_{k\rightarrow \infty}\sum_{i=0}^{N_k}A_{t_i^k,t_{i+1}^k} \text{ in probability.}
\end{align}
Then there exists a map $D \colon \Delta_{[S,T]} \rightarrow L^m$ and a constant $C>0$, such that for all $(s,t) \in \Delta_{[S,T]}$,
\begin{align*}
|\mathcal{A}_t - \mathcal{A}_s -A_{s,t}|\leqslant C \Gamma_1 |t-s|^{\alpha_1} \lambda(s,t)^{\beta_1}+D_{s,t} \text{ a.s.}
\end{align*}
and 
\begin{align*}
\|D_{s,t}\|_{L^m}\leqslant C \Gamma_2 |t-s|^{1/2+\varepsilon}.
\end{align*}

\end{lemma}

\paragraph{\emph{Link between Bm and fBm.}} For $H\in (0,\tfrac{1}{2})$, there exist \red{linear} operators $\mathcal{A}$ and 
$\bar{\mathcal{A}}$ \red{which are formally inverse to one another and can be defined on suitable spaces}. They can both be given in terms of 
fractional integrals and derivatives (see \cite[Th. 11]{Picard}), such that
\begin{align}
\text{if $B$ is a fractional Brownian motion, } W=\mathcal{A}(B) \text{ is a Brownian motion}, 
\label{operatortildeA}\\
\text{if $W$ is a Brownian motion, } B=\bar{\mathcal{A}}(W) \text{ is a fractional Brownian 
motion}. \label{operatorA}
\end{align}
Furthermore $B$ and $W$ generate the same filtration. \red{For yet another representation of such operators see \cite[Equation (7) and (12)]{NualartOuknine}.} 
\begin{lemma}\cite[Lemma B.1]{AnRiTa} \label{lem:operatorcontinuity}
The operator $\mathcal{A}$ continuously maps 
$(\mathcal{C}_{[0,T]},\|\cdot\|_\infty)$ to itself.
\end{lemma}

\begin{definition}
	Let  $\mathbb{G}$ be a filtration. We say that $B$ is a $\mathbb{G}$-fractional Brownian 
	motion if 
$W=\mathcal{A}(B)$ is a $\mathbb{G}$-Brownian motion.
\end{definition}

\paragraph{\emph{Definition of a solution.}}

As the drift in \eqref{eq:skew} is distributional, it is \emph{a priori} not clear how to define a solution. In the literature this is either done as below in Definition~\ref{def:solution} (see \cite{Atetal,AnRiTa,GaleatiGerencser}) or via nonlinear Young integrals (see Appendix~\ref{summary}). The two definitions coincide as long as the solution fulfills a certain regularity (see \cite[Remark 8.5]{GaleatiGerencser} and \cite[Theorem 2.15]{AnRiTa}).

\begin{definition} \label{def:solution}
Let $\beta \in \mathbb{R}$, $b \in \mathcal{B}_\infty^\beta$, $T>0$ and $x_0 \in \mathbb{R}^d$. We call a couple $((X_t)_{t \in [0,T]},(B_t)_{t \in 
[0,T]})$ defined on some filtered probability space  
$(\Omega,\mathcal{G},\mathbb{G},\mathbb{P})$ a \emph{weak solution} to \eqref{eq:skew} on 
$[0,T]$, with 
initial condition $x_0$, if 
\begin{compactitem}
\item $B$ is a $\mathbb{G}$-fBm;

\item $X$ is adapted to $\mathbb{G}$;

\item there exists a process $(K_t)_{t 
\in [0,T]}$ such that, a.s.,
\begin{equation}\label{solution1}
X_t=x_0+K_t+B_t\text{ for all } t \in [0,T] ;
\end{equation}

\item for every sequence $(b^n)_{n\in \mathbb{N}}$ of smooth bounded functions converging to $b$ in $\mathcal{B}^{\beta-}_\infty$, we have that
\begin{equation}\label{approximation2}
       \sup_{t\in [0,T]}\left|\int_0^t b^n(X_r) dr 
       -K_t\right|\limbashaut{\longrightarrow}{n\rightarrow \infty}{\mathbb{P}} 0.
\end{equation}
\end{compactitem}
If the couple is clear from the context, we simply say that  $(X_t)_{t \in [0,T]}$ is a weak 
solution. If $X$ is adapted to the filtration generated by $B$, we call it a \emph{strong solution}.
\end{definition}

\section{Main result and proof} \label{weaksolution}
\red{Throughout this section we work on a fixed time interval $[0,T]$ for some deterministic $T>0$ and consider solutions defined thereon.} Moreover we work under the following assumption:

\begin{assumption} \label{assumption}
\red{Consider $(b,\beta)$ such that $b$ is a \red{measure} with
\begin{equation}
b \in \mathcal{B}^\beta_\infty \text{ for }\beta \in \mathbb{R} \text{ with } \beta>-\frac{1}{2H}.
\end{equation}}
\end{assumption}

Our main result reads as follows:

\begin{theorem} \label{thm:weakexistence}
Let $(b,\beta)$ satisfy Assumption~\ref{assumption}. Then there exists a weak solution to Equation \eqref{eq:skew} \red{in the sense of Definition~\ref{def:solution}}.
\end{theorem}

Combining Remark~\ref{rem:finitemeasure} and Theorem~\ref{thm:weakexistence}, we get Theorem~\ref{cor:measure}.

The remainder of the section is dedicated to proving Theorem~\ref{thm:weakexistence}. This will be done by regularizing the drift, considering the sequence of strong solutions to the corresponding approximated equations and proceeding via a tightness-stability argument. In order to do so we state two \emph{a priori} estimates in Lemma~\ref{lem:apriori}, quantifying the regularization effect of any solution.
\red{First, we state Lemma~\ref{regulINT} since it captures the regularization effect of the fBm. }For a slightly more general statement and the proof see \cite[Lemma 3.4]{Haressetal}.

\begin{lemma} \label{regulINT}
Let  $\gamma \in (-1/(2H),0)$, $m \in [2,\infty)$ and 
$d,e \in \N$. There exists a constant $C>0$ such that for any $0\leqslant S\leqslant T$, 
any $\mathcal{F}_{S}$-measurable random variable $\Xi$ in $\R^e$ and any bounded 
measurable function $f:\mathbb{R}^d\times\R^e \rightarrow \mathbb{R}^d$ fulfilling
\begin{enumerate}[label=(\roman*)]
    \item $\EE\left[ \|f(\cdot,\Xi)\|_{\mathcal{C}^1}^2\right]<\infty$,
    \item $\EE\left[ \|f(\cdot,\Xi)\|_{\mathcal{B}_\infty^\gamma}^m\right]<\infty$,
\end{enumerate}
we have for any $t\in[S,T]$ that
\begin{equation}\label{eq:regulINT}
    \left\|\int_S^t f(B_r,\Xi) \, dr\right\|_{L^m} \leqslant C \, \| \|f(\cdot,\Xi)\|_{\mathcal{B}_\infty^\gamma}\|_{L^m}\, (t-S)^{1+H\gamma}.
\end{equation}
\end{lemma}

Intuitively the regularization effect of any solution $X$ will be similar to the one of a fBm, since it \red{is expected to have a similar oscillatory behaviour}. \red{This is because $X-B$ is expected to be more regular than $B$.} For another perspective on this, note that any solution $X$ has a jointly continuous local time by \cite[Theorem 2.16]{Leetal}, \red{where the authors make use of the fact that $X-B$ is of finite variation}.

\begin{lemma} \label{lem:apriori}
\begin{enumerate}[label=(\alph*)]
\item \label{a}\blue{Let $\beta < 0$}. \red{Let $m \in \mathbb{N}$. Then there exists a constant $C>0$ such that for any $b$ such that $(b,\beta)$ satisfies Assumption~\ref{assumption} and $(s,t) \in \Delta_{[0,T]}$, any weak solution $X$ to \eqref{eq:skew} fulfills} 
\blue{
\begin{align}\label{eq:estimate}
  \|X_{s,t}-B_{s,t}\|_{L^m}&\leqslant 
    C\, 
    \|b\|_{\mathcal{B}^\beta_\infty} \Big(1+\|b\|^{{\frac{-H\beta}{1+H(\beta-1)}}}_{\mathcal{B}_\infty^\beta}\Big)(t-s)^{1+\beta H},
\end{align}
where by assumption ${\frac{-H\beta}{1+H(\beta-1)}} \in (0,\infty)$.}

\item \label{b} \blue{Let $\beta <0$} \red{and $m \in \mathbb{N}$. Let $\delta \in 
(0,1+H\beta)$. Then there exists a constant $C>0$ and there exists a nonnegative random variable $Z$ such that for $\phi,h \in \mathcal{C}_b^\infty(\mathbb{R}^d,\mathbb{R}^d) \cap \mathcal{B}_\infty^\beta$ such that $(\phi,\beta)$ and $(h,\beta)$ satisfy Assumption~\ref{assumption},}
\begin{equation}\label{eq:averagingX2}
    \left|\int_s^t h(X_r) dr\right|\leqslant Z\, |t-s|^\delta, 
\end{equation}
where $X$ is the unique strong solution to \eqref{eq:skew} with drift $\phi$. Moreover, 
\blue{
\begin{align} \label{eq:m}
\|Z\|_{L^m}\leqslant 
C \|h\|_{\mathcal{B}_\infty^\beta}\Big(1+\|\phi\|^{\apriori}_{\mathcal{B}_\infty^\beta}\Big).
\end{align}}
\end{enumerate}
\end{lemma}

\begin{proof}

\ref{a}: The proof is almost identical to \cite[Proposition 5.3]{AnRiTa}. Therein we use nonnegativity of the drift in order to obtain an increasing process for $d=1$. The same can be done here for dimension $d\geqslant 1$ obtaining a process that is increasing in each component and therefore we omit \blue{most} details (for a similar argument see the proof of \ref{b}).

\blue{The proof of \cite[Proposition 5.3]{AnRiTa} gives existence of a constant $\tilde{C}>0$ such that for $(s,t) \in \Delta_{[0,T]}$ fulfilling $\tilde{C}\|b\|_{\mathcal{B}_\infty^\beta} (t-s)^{1+H(\beta-1)}<1/2$,
\begin{align}\label{eq:aprioriK}
    \|K_t-K_s\|_{L^m}\leqslant \tilde{C}\|b\|_{\mathcal{B}_\infty^\beta} (t-s)^{1+H\beta}.
\end{align}}

\blue{Choose $l = (3\tilde{C}	\|b\|_{\mathcal{B}^\beta_\infty})^{\frac{1}{H(1-\beta)-1}}$ so that $\tilde{C} 
	\|b\|_{\mathcal{B}^\beta_\infty}l^{1+H(\beta-1)}<1/2$.
	Let $u \in [0,T]$. In particular, by \eqref{eq:aprioriK},
\begin{align} \label{eq:beforeiteration2}
    [K]_{\COO{1+H\beta}{m}{\infty}{[u,(u+l) \wedge T]}}\leqslant \tilde{C} \|b\|_{\mathcal{B}_\infty^\beta}.
\end{align}}
\textcolor{black}{If $l>T$, then \eqref{eq:estimate} follows from \eqref{eq:beforeiteration2}. Hence, we  assume $l\leqslant T$. To obtain \eqref{eq:estimate}, we will iteratively apply inequality 
\eqref{eq:beforeiteration2}. Let $(s,t) \in \Delta_{[0,T]}$ be arbitrary. Let 
\(N = \lceil T/l\rceil\) 
 and let the sequence 
$(s_k)_{k=0}^N$ be defined by $s_{k} = s+ k(t-s)/N$. 
By triangle inequality and \eqref{eq:beforeiteration2}, we get
\begin{align*}
    \|K_t-K_s\|_{L^m}&\leqslant \sum_{k=1}^N \|K_{s_k}-K_{s_{k-1}}\|_{L^m}\\
    &\leqslant \tilde{C} \|b\|_{\mathcal{B}_\infty^\beta} \sum_{k=1}^N (s_k-s_{k-1})^{1+H\beta}\\
    &\leqslant \tilde{C} \|b\|_{\mathcal{B}_\infty^\beta} \sum_{k=1}^N \left(\frac{t-s}{N}\right)^{1+H\beta}\\
    &\leqslant \tilde{C} \|b\|_{\mathcal{B}_\infty^\beta} N^{-H\beta} (t-s)^{1+H\beta}.
\end{align*}}
\textcolor{black}{Using \(N \leqslant 1+\frac{T}{l} \leqslant 2 \frac{T}{l}\leqslant C\|b\|_{\mathcal{B}^\beta_\infty}^{\frac{1}{1+H(\beta-1)}}$, it follows that
\begin{align*}
\|b\|_{\mathcal{B}^\beta_\infty} N^{-H \beta}\leqslant C\|b\|_{\mathcal{B}^\beta_\infty}\|b\|_{\mathcal{B}^\beta_\infty}^{\frac{-H\beta}{1+H(\beta-1)}}
\end{align*}
and therefore \eqref{eq:estimate}.}

\ref{b}: By nonnegativity of $\phi$, $K=X-B$ is monotone in each component. In particular $(v,w)\mapsto |K_w-K_v|$ defines a random control \red{as the choice of norm on $\mathbb{R}^d$ gives superadditivity (see list of notations)}.
For $(s,t) \in \Delta_{[0,T]}$ let
\begin{align*}
    A_{s,t}&:=\int_{s}^{t} h(B_r+K_s)\, dr.
\end{align*}

We apply Lemma~\ref{lem:stsrandomcontrol} for $K^h=\mathcal{A}$ defined by $K^h_{\cdot}=\int_0^{\cdot} h(X_r) dr$. In order to see that all conditions are fulfilled, 
we will show the following for $ u \in [s,t]$ and some constant $C>0$ independent of $s,t$ and $u$:
\begin{enumerate} [label=(\roman*)]
\item \label{enum:1ALT2} $\|A_{s,t}\|_{L^m}\leqslant C \|h\|_{\mathcal{B}_\infty^\beta}(t-s)^{1+H\beta}$;

\item \label{enum:3ALT2} $|\EE^u[\delta A_{s,u,t}]|\leqslant C \|h\|_{\mathcal{B}^\beta_\infty} |K_u-K_{s}|(t-u)^{H(\beta-1)+1}$;

\item \label{enum:2ALT2} $\sum_{i=0}^{N_n-1} 
A_{t^n_i,t^n_{i+1}}\overset{a.s.}{\longrightarrow} K^h_t$ along any sequence of partitions $\Pi_n=\{t_i^n\}_{i=0}^{N_n}$ of $[0,t]$ 
with mesh converging to $0$. 
\end{enumerate}
Notice that $1+H\beta>1/2$ and hence \ref{enum:1ALT2} gives \eqref{sts2}. Furthermore, \ref{enum:3ALT2} 
gives \eqref{sts1} for $\alpha_1=H(\beta-1)+1>1/2-H\geqslant 0$, 
$\beta_1=1$ and $\lambda(s,t):=|K_t-K_s|$.

First, assume that \ref{enum:1ALT2}-\ref{enum:3ALT2}-\ref{enum:2ALT2} hold true. Then by Lemma~\ref{lem:stsrandomcontrol}, there exists a process $D$ such that
\begin{align}\label{eq:KtnALT2}
    |K^h_{t}-K^h_{s}-A_{s,t}|\leqslant C\|h\|_{\mathcal{B}_\infty^\beta} |K_{t}-K_{s}| (t-s)^{H(\beta-1)+1} + D_{s,t},
\end{align}
with $\|D_{s,t}\|_{L^m} \leqslant C \|h\|_{\mathcal{B}_\infty^\beta} (t-s)^{1+H\beta}$.
Hence, by \ref{enum:1ALT2} and Lemma~\ref{lem:apriori}\ref{a} and as $H(\beta-1)+1>0$
\begin{align*}
\|K^h_{t}-K^h_{s}\|_{L^m} &\leqslant C\|h\|_{\mathcal{B}_\infty^\beta} \left(\|K_t-K_s\|_{L^m} (t-s)^{{H(\beta-1)+1}} +  (t-s)^{1+H\beta}\right)\\
&\leqslant C \|h\|_{\mathcal{B}^\beta_\infty}(1+\blue{\|\phi\|^{\apriori}_{\mathcal{B}^\beta_\infty})} (t-s)^{1+H\beta}.
\end{align*}

The result now follows from Kolmogorov's continuity theorem \red{after choosing $m$ large enough, which is no restriction as \eqref{eq:m} then also holds true for any smaller choice of $m$}.

Let us now verify \ref{enum:1ALT2}-\ref{enum:3ALT2}-\ref{enum:2ALT2}.
 
Proof of \ref{enum:1ALT2}: By Lemma~\ref{regulINT} applied to $\Xi=K_{s}$ and $f(z,x) = h(z+x)$, we have 
\begin{align*}
 \|A_{s,t}\|_{L^m}&\leqslant C \big\| \|h(\cdot+K_s)\|_{\mathcal{B}_\infty^{\beta}} \big\|_{L^m} (t-s)^{1+H\beta} .
\end{align*}
Using that $\|h(\cdot+K_s)\|_{\mathcal{B}^\beta_\infty}=\|h\|_{\mathcal{B}^\beta_\infty}$ (see \cite[Lemma A.2]{Atetal} for $d=1$, which easily generalizes to $d>1$), we thus get
\begin{align} \label{regulAnALT2}
    \|A_{s,t}\|_{L^m}&\leqslant C \|h\|_{\mathcal{B}_\infty^{\beta}} (t-s)^{1+H\beta}.
\end{align}

Proof of \ref{enum:3ALT2}: From the local nondeterminism property of fractional Brownian motion (see 
Lemma 7.1 in \cite{Pitt}), we have that 
\begin{equation} \label{eq:LND}
    \sigma^2_{s,t}\geqslant C(t-s)^{2H},
\end{equation}
where
\begin{equation*}
\sigma^2_{s,t} I_d=\text{Var}(B_t-\EE^s[B_t]).
\end{equation*}

Then by \cite[Lemma 3.3]{Haressetal} applied to $\Xi = (K_{s},K_{u})$ and 
$f(z,(x_{1},x_{2}))=h(z+x_{1}) - h(z+x_{2})$, we obtain
\begin{align*}
    |\EE^u[\delta A_{s,u,t}]|&= \Big| \int_u^t \EE^u [h(B_r+K_s)-h(B_r+K_u)] dr\Big| \\
    &= \Big|\int_u^t G_{\sigma^2_{u,r}}h(\EE^u[B_r]+K_s)-G_{\sigma^2_{u,r}}h(\EE^u[B_r]+K_u) dr\Big|\\
    &\leqslant \int_u^t \|G_{\sigma^2_{u,r}}h\|_{\mathcal{C}^1}|K_u-K_s| dr .
   \end{align*}

We now use \cite[Lemma A.3 (iv)]{Atetal} which again easily generalizes to $d>1$. Note that it can be used as $\beta<0$. Hence
\begin{equation*}
\|G_{\sigma^2_{u,r}}h\|_{\mathcal{C}^1}\leqslant C\|h\|_{\mathcal{B}_\infty^\beta} (\sigma_{u,r}^2)^{(\beta-1)/2}.
\end{equation*}
The above, \eqref{eq:LND} and using that $H(\beta-1)>-1$ to ensure integrability gives
\begin{align*}
    |\EE^u[\delta A_{s,u,t}]|
    &\leqslant C \int_u^t |r-u|^{H(\beta-1)} \|h\|_{\mathcal{B}^\beta_\infty} |K_u-K_s|dr \\
    &\leqslant C \|h\|_{\mathcal{B}^\beta_\infty} |K_u-K_s|(t-u)^{H(\beta-1)+1}.
\end{align*}

Proof of \ref{enum:2ALT2}: For a sequence $\Pi_n=\{t_i^n\}_{i=0}^{N_n}$ of partitions of 
$[0,t]$ with mesh size going to $0$, we have
\begin{align*}
    |K_t^h - \sum_{i=0}^{N_{n}-1} A_{t_i^n,t^n_{i+1}}| &\leqslant \sum_i \int_{t^n_i}^{t^n_{i+1}} |h(B_r+K_r)-h(B_r+K_{t_i^n})|dr\\
    &\leqslant \sum_i \int_{t^n_i}^{t^n_{i+1}} \|h\|_{\mathcal{C}^1} |K_r-K_{t^n_i}| dr \\
    &\leqslant \sum_i \|h\|_{\mathcal{C}^1} (t^n_{i+1}-t^n_i) |K_{t^n_{i+1}}-K_{t^n_i}|\\
    &\leqslant \|h\|_{\mathcal{C}^1} |\Pi_n| |K_t-K_0|\overset{n \rightarrow 
    \infty}{\longrightarrow} 0 \text{ a.s.}
\end{align*}
\end{proof}

Below we state the two propositions that ensure tightness and stability of the approximation scheme. The proofs are similar to the ones of \cite[Proposition 7.4 and Proposition 7.6]{AnRiTa}. The only differences are allowing for $d\geqslant 1$ instead of $d=1$ and that Assumption~\ref{assumption} is weaker than the corresponding assumption in there. The latter is due to the crucial Lemma~\ref{lem:apriori}\ref{b} being stated under Assumption~\ref{assumption}. For the reader's convenience we prove both statements.

\begin{proposition}[Tightness] \label{prop:tightness}
Assume $(b,\beta)$ fulfills Assumption~\ref{assumption}. Let \((b^n)_{n \in \mathbb{N}}\) 
be a sequence of 
smooth bounded functions converging to $b$ in $\mathcal{B}_\infty^{\beta-}$. For $n 
\in \mathbb{N}$, let $X^{n}$ be the strong solution to  
\eqref{eq:skew} with initial condition $x_0$ and drift $b^n$. Then there exists a 
subsequence $(n_k)_{k\in \mathbb{N}}$ such that $(X^{n_k},B)_{k \in 
\mathbb{N}}$ converges weakly in the space $[\mathcal{C}_{[0,T]}]^2$.
\end{proposition}

\begin{proposition}[Stability] \label{prop:stability}
	Assume $(b,\beta)$ fulfills Assumption~\ref{assumption}. Let $(\tilde{b}^k)_{k \in \mathbb{N}}$ be a sequence of smooth bounded functions 
	converging to $b$ in $\mathcal{B}_\infty^{\beta-}$. Let $\tilde{B}^k$ have the same law as 
	$B$ \red{and assume all $\tilde{B}^k$ are defined on the same filtered probability space}. 
	We consider $\tilde{X}^k$ to be the unique strong solution to \eqref{eq:skew} for $B=\tilde{B}^k$, drift $\tilde{b}^k$ and initial 
	condition 
	$x_0$. 
	We assume that  there exist stochastic processes $\tilde{X},\tilde{B}: [0,T]  
	\rightarrow \mathbb{R}^d$ such that $(\tilde{X}^k,\tilde{B}^k)_{k \in \mathbb{N}}$ converges 
	to 
	$(\tilde{X},\tilde{B})$ on $[\mathcal{C}_{[0,T]}]^2$ in probability. Then $\tilde{X}$ fulfills 
	\eqref{solution1} and \eqref{approximation2} from Definition~\ref{def:solution}.
\end{proposition}

\begin{proof}[Proof of Proposition~\ref{prop:tightness}]
\blue{Assume w.l.o.g. that $\beta<0$.} Let $K^{n}_t:=\int_0^t b^n(X^{n}_r)dr$. 
For $M>0$, let  
\begin{align*}
    A_M:=\{f \in \mathcal{C}_{[0,T]}: f(0)=0,\ |f(t)-f(s)|\leqslant M (t-s)^{1+H\beta},\ \forall (s,t) \in \Delta_{[0,T]}\}.
\end{align*}
By Arzel\`a-Ascoli, $A_M$ is compact in $\mathcal{C}_{[0,T]}$. 
Applying Lemma~\ref{lem:apriori}\ref{a} 
and Markov's inequality we get
\begin{align*}
    \mathbb{P}(K^{n} \notin A_M) &\leqslant \mathbb{P}(\exists (s,t) \in 
    \Delta_{[0,T]}:|K^{n}_{s,t}|> M (t-s)^{1+H\beta})\\
&\leqslant C\, \sup_{n \in \mathbb{N}} \blue{
    \, 
    (1+\|b^n\|^{\apriori}_{\mathcal{B}_\infty^\beta})} \, M^{-1}.
\end{align*}
Hence, $(K^{n})_{n \in \mathbb{N}}$  is tight in $\mathcal{C}_{[0,T]}$ and  therefore $(K^{n},B)_{n \in \mathbb{N}}$ is tight in $[\mathcal{C}_{[0,T]}]^2$. Using Prokhorov's 
Theorem, there exists a subsequence $(n_k)_{k \in \mathbb{N}}$ such that $(K^{{n_k}},B)_{k 
\in \mathbb{N}}$ converges weakly in the space $[\mathcal{C}_{[0,T]}]^2$, and so does 
$(X^{{n_k}},B)_{k \in \mathbb{N}}$.
\end{proof}

\begin{proof}[Proof of Proposition~\ref{prop:stability}]
\blue{Assume w.l.o.g. that $\beta<0$ and \(X_0 = 0\).} Let
		\(\tilde{K}:=\tilde{X}-\tilde{B}\), 
	so that \eqref{solution1} is verified. Let $(b^n)_{n \in 
		\mathbb{N}}$ be any sequence of smooth bounded functions converging to $b$ in 
	$\mathcal{B}_\infty^{\beta-}$. 
	To check that $\tilde{K}$ and $\tilde{X}$ fulfill \eqref{approximation2} 
	from Definition~\ref{def:solution}, we have to show that 
	\begin{align}
		\lim_{n \rightarrow \infty} \sup_{t \in [0,T]} \left|\int_0^t b^n(\tilde{X}_r) 
		dr-\tilde{K}_t\right|=0 
		\text{ in probability}. \label{convergence}
	\end{align}	
	By the triangle inequality we have for $k,n \in \mathbb{N}$ and $t \in [0,T]$,
	\begin{align} \label{I1I2I3}
		\left|\int_0^t b^n(\tilde{X}_r) dr-\tilde{K}_t\right|\leqslant &\left|\int_0^t b^n(\tilde{X}_r) 
		dr-\int_0^t 
		b^n(\tilde{X}_r^k)dr\right|+\left|\int_0^t b^n(\tilde{X}_r^k)dr-\int_0^t 
		\tilde{b}^k(\tilde{X}_r^k)dr\right|\nonumber\\
		&+\left|\int_0^t \tilde{b}^k(\tilde{X}_r^k)dr-\tilde{K}_t\right|=:A_1+A_2+A_3.
	\end{align}
	We show that all three summands in \eqref{I1I2I3} 
	converge to $0$ uniformly on $[0,T]$ in probability as $n\to \infty$, 
	choosing $k=k(n)$ accordingly.
	
	$\mathbf{A_1}$: Notice that
	\begin{align*}
		\left|\int_0^t b^n(\tilde{X}_r) dr-\int_0^t b^n(\tilde{X}_r^k)dr\right|&\leqslant 
		\|b^n\|_{\mathcal{C}^1} \int_0^t |\tilde{X}_r-\tilde{X}_r^k| dr\\
		&\leqslant \|b^n\|_{\mathcal{C}^1}\,  T \sup_{t \in [0,T]} 
		 |\tilde{X}_t-\tilde{X}_t^k|.
	\end{align*}
	The result follows as for any $\varepsilon>0$, by assumption, we can choose a sequence $(k(n))_{n \in 
	\mathbb{N}}$ such that 
	\begin{align*}
		\mathbb{P}\Big(\|b^n\|_{\mathcal{C}^1}\, T\sup_{t \in [0,T]} 
		|\tilde{X}_t-\tilde{X}_t^{k(n)}| > \varepsilon\Big)< \frac{1}{n}, ~ \forall n \in 
		\mathbb{N}.
	\end{align*}

$\mathbf{A_2}$: Let $\beta^\prime \in (-1/(2H),\beta)$. 
	By 
	Lemma~\ref{lem:apriori}\ref{b} applied to $\tilde{X}^k$, $h=b^n-\tilde{b}^k$ and $\beta^\prime$ 
	instead of 
	$\beta$, there exists a random variable 
	$Z_{n,k}$ with
	\begin{align} \label{eq:expectation}
		\EE[Z_{n,k}]&\leqslant C\, 
			\|b^n-\tilde{b}^k\|_{\mathcal{B}_\infty^{\beta^\prime}}(1+\blue{\|\tilde{b}^k\|^{\aprioritwo}_{\mathcal{B}_\infty^{\beta^\prime}}})\nonumber\\
		& \leqslant C\, 
			(\|b^n-b\|_{\mathcal{B}_\infty^{\beta^\prime}}+\|\tilde{b}^k-b\|_{\mathcal{B}_\infty^{\beta^\prime}})
			 \, 
			(1+\sup_{m \in \mathbb{N}}\blue{\|\tilde{b}^m\|^{\aprioritwo}_{\mathcal{B}_\infty^{\beta^\prime}}}),
	\end{align}
	for $C$ independent of $k,n$, such that we have
	\begin{align*}
		\sup_{t \in [0,T]}\left|\int_0^t b^n(\tilde{X}_r^k)dr-\int_0^t 
		\tilde{b}^k(\tilde{X}_r^k)dr\right|\leqslant 
		Z_{n,k}(1+T).
	\end{align*}
	Using Markov's inequality and \eqref{eq:expectation} we obtain that
	\begin{align*}
		\mathbb{P}&\left(\sup_{t \in [0,T]} \left|\int_0^t b^n(\tilde{X}_r^k)dr-\int_0^t 
		\tilde{b}^k(\tilde{X}_r^k)dr\right|>\varepsilon\right)\leqslant \varepsilon^{-1}\, \EE[Z_{n,k}] 
		\, (1+T)\\
		&\qquad \qquad \leqslant C\, \varepsilon^{-1}\, (1+T)\, 
			(\|b^n-b\|_{\mathcal{B}_\infty^{\beta^\prime}}+\|\tilde{b}^k-b\|_{\mathcal{B}_\infty^{\beta^\prime}})
			 \, 
			(1+\sup_{m \in \mathbb{N}}\blue{\|\tilde{b}^m\|_{\mathcal{B}_\infty^{\beta^\prime}}^{\aprioritwo}}).
	\end{align*}
	Choosing $k=k(n)$ as before gives the convergence.

$\mathbf{A_3}$: Recall that $\tilde{X}^k_t=\int_{0}^t 
		\tilde{b}^k(\tilde{X}^k_r) 
	dr+\tilde{B}^k_t$. Hence,
	\begin{align*}
		\sup_{t \in [0,T]} \left|\int_0^t \tilde{b}^k(\tilde{X}^k_r) dr-\tilde{K}_t\right| \leqslant 
		\sup_{t \in 
		[0,T]}(|\tilde{X}^k_t-\tilde{X}_t|+|\tilde{B}_t^k-\tilde{B}_t|).
	\end{align*}
	Since by assumption $(\tilde{X}^k,\tilde{B}^k)_{k \in \mathbb{N}}$ converges to 
	$(\tilde{X},\tilde{B})$ on 
	$[\mathcal{C}_{[0,T]}]^2$ in probability, we get the result.
\end{proof}

\begin{proof}[Proof of Theorem~\ref{thm:weakexistence}]

Let $(b^n)_{n \in \mathbb{N}}$ be a sequence of 
smooth bounded functions converging to $b$ in $\mathcal{B}_\infty^{\beta-}$. 
By Proposition~\ref{prop:tightness}, 
there exists a subsequence $(n_k)_{k \in \mathbb{N}}$ such that $(X^{n_k}, B)_{k 
\in \mathbb{N}}$ converges weakly in $[\mathcal{C}_{[0,T]}]^2$.  
W.l.o.g., 
we assume that $(X^{n}, B)_{n \in \mathbb{N}}$ converges weakly. By 
Skorokhod's representation Theorem, there exists a sequence of random variables 
$(Y^{n},\hat{B}^n)_{n \in \mathbb{N}}$ defined on a common probability space 
$(\hat{\Omega},\hat{\mathcal{F}},\hat{P})$, such that
\begin{align} \label{samelaw}
    \text{Law}(Y^{n},\hat{B}^n)=\text{Law}(X^{n}, B), \ \forall n \in \mathbb{N},
\end{align}
and  $(Y^{n},\hat{B}^n)$ converges a.s. to some $(Y,\hat{B})$ in $[\mathcal{C}_{[0,T]}]^2$.
As $X^{n}$ solves the SDE \eqref{eq:skew} with drift $b^n$, we know by \eqref{samelaw} that $Y^{n}$ also solves 
\eqref{eq:skew} with drift $b^n$ and $\hat{B}^n$ 
instead of $B$. Since $X^{n}$ is a strong solution, we have that $X^{n}$ 
is adapted to $\mathbb{F}^B$. Hence by \eqref{samelaw}, we know that $Y^{n}$ is adapted to $\mathbb{F}^{\hat{B}^n}$ as the conditional laws of 
$Y^{n}$ and $X^{n}$ agree and therefore $Y^n$ is a strong solution to 
\eqref{eq:skew} with $\hat{B}^n$ instead of $B$.

By Proposition~\ref{prop:stability}, we know that $Y$ fulfills (\ref{solution1}) and 
(\ref{approximation2}) from Definition~\ref{def:solution} with $\hat{B}$ instead of $B$. Clearly $Y$ is adapted to the filtration $\hat{\mathbb{F}}$ defined by 
$\hat{\mathcal{F}}_t:= \sigma(Y_{s},\hat{B}_{s},s \in [0,t])$. 
By \eqref{operatorA} and \eqref{operatortildeA}, we have
\begin{align}
\hat{B}^n&=\bar{\mathcal{A}}(W^n) \text{ a.s. and}\label{representation1}\\
W^n&=\mathcal{A}(\hat{B}^n) \text{ a.s.},\label{representation2}
\end{align}
for a sequence $W^n$ of Brownian motions with 
$\mathbb{F}^{\hat{B}^n}=\mathbb{F}^{W^n}$. By definition, for $(s,t) \in \Delta_{[0,T]}$, $W^n_t-W^n_s$ is 
independent of $\mathcal{F}^{W^n}_s=\mathcal{F}^{\hat{B}^n}_s= 
\sigma(Y^{n}_{r},\hat{B}^n_{r}, r \in [0,s])$. By \eqref{representation1}, 
\eqref{representation2}, Lemma~\ref{lem:operatorcontinuity} and the 
a.s.-convergence of $\hat{B}^n$, we get that $W^n$ converges a.s. uniformly on 
$[0,T]$ to a Brownian motion $W$ such that $\hat{B}=\bar{\mathcal{A}}(W)$ and $\hat{B}$ and $W$ 
generate the same filtration. Hence, we can deduce that $W_t-W_s$ is independent of 
$\hat{\mathcal{F}}_s$ and so $W$ is an $\hat{\mathbb{F}}$-Bm.
Therefore, $\hat{B}$ is a $\hat{\mathbb{F}}$-fBm and $Y$ is adapted to $\hat{\mathbb{F}}$. Hence $Y$ is a weak solution. 
\end{proof}

\begin{appendices}

\section{Previous research on \eqref{eq:skew}}\label{summary}

In this section we provide an overview of the results in recent years for 
Equation \eqref{eq:skew} in case of $B$ being a fractional Brownian motion with $H \neq 1/2$ 
(see Theorem~\ref{thm:overview}). A by now classical approach is to rewrite \eqref{eq:skew} 
as a nonlinear Young integral. Hence, we briefly describe this approach first.

Consider the rewritten SDE
\begin{align} \label{rewritten}
X_t=x_0+\int_0^t b(X_s) ds + B_t \iff \tilde{X}_t=x_0+\int_0^t b(\tilde{X}_s+B_s) ds, \quad t \in [0,T], x_0 \in \mathbb{R}^d.
\end{align}

For a continuous bounded vector valued function $b:\mathbb{R}^d \rightarrow \mathbb{R}^d$, define the averaging operator $T^B$ by 
\begin{align} \label{traditionalaveraging}
    T^\mathit{B}_t b(x):= \int_0^t b(x+B_r) \, dr, ~ \mbox{ for } (t,x) \in [0,T]\times \mathbb{R}^d.
\end{align} 
Then can rewrite the integral on the right hand side of \eqref{rewritten} by
\begin{align} \label{eq:reformulation}
    \int_0^t b(\tilde{X}_r+B_r)\, dr &= \lim_{n \rightarrow \infty} \sum_{i=1}^{N_n-1} 
    \int_{t_i^n}^{t_{i+1}^n} b(\tilde{X}_{t_i^n}+ B_r) \, dr \nonumber\\
    &= \lim_{n \rightarrow \infty} \sum_{i=1}^{N_n-1} T^{B}_{t_i^n,t_{i+1}^n} 
    b(\tilde{X}_{t_i}^n) \nonumber\\
    &= \int_0^t T_{dr}^{B} b(\tilde{X}_r) ,
\end{align}
with the final equality being only formal at this point. One can give a rigorous definition of this so-called nonlinear Young integral. For a detailed review see \cite{Galeati}. 

The averaging operator $T^{B} b$ can also be written as a convolution against the occupation 
measure of the noise. This viewpoint was taken in \cite{HarangPerkowski,AnRiTa}. Either way, 
for $H$ sufficiently small (i.e. $B$ oscillating sufficiently fast) one can extend the definition of 
the averaging operator $T^B b$ to singular or even distributional $b$. Then, given that 
expression in \eqref{eq:reformulation} is well-defined, a solution to \eqref{eq:skew} can be 
defined to be a solution to the corresponding nonlinear Young integral equation. Note that the 
above has a deterministic flavor as properties for each fractional Brownian path are needed. Stochastic sewing might still be useful to obtain regularity of the averaging operator (see \cite{Le,GaleatiGerencser, HarangPerkowski}). If the solution fulfills a regularity condition, the definition of a solution via nonlinear 
Young integrals and Definition~\ref{def:solution} coincide (see \cite[Remark 
8.5]{GaleatiGerencser} and \cite[Theorem 2.15]{AnRiTa}).

The following Theorem gives an overview of the developments for Equation \eqref{eq:skew} and related equations in recent years. In \ref{(1)} as well as partly in \ref{(2)} and \ref{(3)} a nonlinear Young integral approach as described above was followed. To ensure readability, some results do not represent the full scope of the actual results proven. Below we partly also allow for time-dependent drift $b$.

\begin{theorem} \label{thm:overview}
\begin{enumerate}[label=(\alph*)]
\item \label{(1)}\cite[Theorem 1.9]{CatellierGubinelli} combined with \cite[Theorem 3.13]{HarangGaleati}: Let $b \in \mathcal{B}^\beta_\infty$ for $\beta>1-1/(2H)$. Then there exists a strong solution to \eqref{eq:skew} and path-by-path uniqueness holds (i.e. uniqueness to the integral equation for almost every realization of the noise, giving a stronger notion of uniqueness than the classical notion of pathwise uniqueness)
\item \label{(2)}\cite[Theorem 2.8, Theorem 2.10 and Corollary 2.6]{AnRiTa}: Let $d=1$. For $H<\sqrt{2}-1$, there exists a weak solution for a finite measure $b$. Additionally, for $b \in \mathcal{B}^\beta_\infty$ with $\beta>1/2-1/(2H)$, there exists a weak solution. \red{Pathwise uniqueness and strong existence is shown for $b \in \mathcal{B}^\beta_p$ for $\beta \in \mathbb{R}$ and $p \in [1,\infty]$ with
\begin{align*}
\beta> 1-\frac{1}{2H} \text{ and } \beta-\frac{1}{p}\geqslant 1-\frac{1}{2H}.
\end{align*}}
\item \label{Haress} \red{\cite[Theorem 2.3]{Haressetal}: Generalization of the latter two results in \ref{(2)} to $d$ dimensions; i.e. there exists a weak solution for $b \in \mathcal{B}^\beta_\infty$ with $\beta>1/2-1/(2H)$ and pathwise uniqueness and strong existence holds for $b \in \mathcal{B}^\beta_p$ with 
\begin{align*}
\beta> 1-\frac{1}{2H} \text{ and } \beta-d/p\geqslant 1-\frac{1}{2H}.
\end{align*}}
\item \label{(3)}\cite[Theorem 1.4]{GaleatiGerencser}: Strong existence, path-by-path uniqueness, Malliavin differentiability and existence of a flow for time-dependent drift $b \in L^q([0,T],\mathcal{B}^\beta_\infty)$ for $q \in (1,2]$ and $\beta>1-1/(Hq^\prime)$. Additionally weak existence for $b \in L^q([0,T],\mathcal{B}^\beta_\infty)$ for $q \in (2,\infty]$ and 
\begin{align*}
\beta>\Big(1-\frac{1}{Hq^\prime}\Big)\vee \Big(\frac{1}{2}-\frac{1}{2H}\Big)
\end{align*}
is shown. \red{In both of the above $q^\prime$ fulfills $1/q+1/{q^\prime}=1$.}
\item \label{(4)}\cite[Theorem 2.11 and Theorem 2.14]{Leetal}: For a finite measure $b$: Existence of a weak solution for $H<1/(1+d)$ for any $d \in \mathbb{N}$; pathwise uniqueness and existence of a strong solution for $H<(\sqrt{13}-3)/2$ and $d=1$.
\end{enumerate}
\end{theorem}

\begin{remark}
In particular, Theorem~\ref{cor:measure} can be seen as an extension of the existence result in \ref{(2)} and as a slightly weaker result than \ref{(4)}.
\end{remark}

Finally let us mention that also similar equations were investigated, such as the case of local time drift \cite{AmineEtAl,Banos}, distribution-dependent drift \cite{HarangGaleati}, multiplicative fractional noise \cite{DareiotisGerencser}, L\'evy noise \cite{KrempPerkowski}, ``infinitely regularizing" noises \cite{HarangPerkowski} and regular noise \cite{Gerencser}.

\section{Besov spaces} \label{sec:Besov}

In this section we briefly recall the definition of Besov spaces. \red{For a complete presentation see \cite{BaDaCh}.}

\begin{definition}[Partition of unity] \label{def:POF}

Let $\chi, \rho \in C^\infty(\mathbb{R}^d,\mathbb{R})$ be radial functions and for \(j\geqslant 0\), 
$\rho_j(x)=\rho(2^{-j}x)$.
We assume that $\chi$ is supported on a ball around $0$ and $\rho$ is supported on an annulus. Moreover, we have 
\begin{align}
\chi + \sum_{j\geqslant 0} \rho_j &\equiv 1,\\
\supp(\chi)\cap \supp(\rho_j)&=\emptyset, \ \forall j\geqslant 1,\\
\supp(\rho_j)\cap \supp(\rho_i)&=\emptyset, \text{ if } |i-j|\geqslant 2.
\end{align}
Then we call the pair $(\chi,\rho)$ a partition of unity. 
\end{definition}
Existence of a partition of unity  is proven in \cite[Prop. 2.10]{BaDaCh}. 
Throughout the paper such a partition is fixed.

\begin{definition}[Littlewood-Paley blocks] \label{def:LP}
Let $f$ be an $\mathbb{R}^d$-valued tempered distribution. We define its $j$-th Littlewood-Paley block by
\[
    \boldsymbol{\Delta}_j f=\begin{dcases}
        \mathcal{F}^{-1}(\rho_j \mathcal{F}(f))~ \text{ for } j\geqslant 0\, , \\
        \mathcal{F}^{-1}(\chi \mathcal{F}(f))~ \text{ for } j=-1\, ,\\
        0~ \text{ for } j\leqslant -2,
    \end{dcases}
\]
where \(\mathcal{F}\) and \(\mathcal{F}^{-1}\) denote the Fourier transform and its inverse.
\end{definition}

\begin{definition}
For $s \in \mathbb{R}$ and $p \in [1,\infty]$, let the nonhomogeneous Besov 
space $\mathcal{B}_{p}^s$ be the space of $\mathbb{R}^d$-valued tempered distributions $u$ such that
\begin{align*}
    \|u\|_{\mathcal{B}_{p}^s}:= \sup_{j \in \mathbb{Z}}2^{js}\|\boldsymbol{\Delta}_j u\|_{L^p(\mathbb{R}^d)}<\infty.
\end{align*}

\end{definition}

\begin{remark}\label{rem:finitemeasure}
Note that any finite measure lies in $\mathcal{B}^0_1$ by similar computations as in \cite[Proposition 2.39]{BaDaCh}. Therefore, after an embedding of Besov spaces \red{(see \cite[Proposition 2.71]{BaDaCh})}, it lies in $\mathcal{B}^{-d}_\infty$ as well.
\end{remark}

\end{appendices}

\begin{acknowledgement}
I would like to kindly thank the anonymous referee for careful reading and remarks that helped improving the paper. I thank both my advisors Alexandre Richard and Etienne Tanr\'e for valuable suggestions and Lucio Galeati for pointing out an error in an estimate.
\end{acknowledgement}

\end{document}